\documentclass{article}
\usepackage{amsmath,amsthm,amssymb,amscd}
\newtheorem{thm}{Theorem}[section]

\newtheorem{prp}[thm]{Proposition}
\newtheorem{crl}[thm]{Corollary}
\theoremstyle{definition}

\theoremstyle{remark}
\newtheorem*{rmk}{Remark}

\title{On model structure for coreflective subcategories of a model category} 

\author{Tadayuki Haraguchi} 
\date{\empty}
\begin{document}
\maketitle
\section{Introduction}

Let $\bf C$ be a coreflective subcategory of a cofibrantly generated model category $\bf D$.
In this paper we show that under suitable conditions
$\bf C$ admits a cofibrantly generated model structure which is left Quillen adjunct to the model structure on $\bf D$.
As an application,
we prove that well-known convenient categories of topological spaces,
such as $k$-spaces,
compactly generated spaces,
and $\Delta$-generated spaces \cite{DN} (called numerically generated in \cite{KKH}) admit a finitely generated model structure which is Quillen equivalent to the standard model structure on
the category $\bf Top$ of topological spaces.
\section{Coreflective subcategories of a model category}
Let $\bf D$ be a cofibrantly generated model category {\cite[2.1.17]{Hov}}
with generating cofibrations $I$,
generating trivial cofibrations $J$ and the class of weak equivalences $W_{\bf D}$.
If the domains and codomains of $I$ and $J$ are finite relative to
$I$-cell {\cite[2.1.4]{Hov}}, then $\bf D$ is said to be finitely
generated.

Recall that a subcategory $\bf C$ of $\bf D$ is said to be
coreflective if the inclusion functor $i \colon {\bf C} \to {\bf D}$
has a right adjoint $G \colon {\bf D} \to {\bf C}$, so that there is a
natural isomorphism $\varphi \colon {\rm Hom}_{\bf D}(X,Y) \to {\rm
  Hom}_{\bf C}(X,GY)$.  The counit of this adjunction $\epsilon
\colon GY \to Y\ (Y \in {\bf D})$ is called the coreflection arrow.
\begin{thm}\label{thm:cofibrant sub}
  Let $\bf C$ be a coreflective subcategory of a cofibrantly generated
  model category $\bf D$ which is complete and cocomplete. %
  Suppose that the unit of the adjunction $\eta \colon X \to GX$ is a
  natural isomorphism, and that the classes $I$ and $J$ of
  cofibrations and trivial cofibrations in $\bf D$ are contained in
  $\bf C$. %
  Then $\bf C$ has a cofibrantly generated model structure with $I$ as
  the set of generating cofibrations, $J$ as the set of generating
  trivial cofibrations, and $W_{\bf C}$ as the class of weak
  equivalences, %
  where $W_{\bf C}$ is the class of all weak equivalences contained in
  $\bf C$. %
  If $\bf D$ is finitely generated, then so is $\bf C$. %
  Moreover, the adjunction $(i,G,\varphi) \colon {\bf C} \to {\bf D}$
  is a Quillen adjunction in the sense of {\cite[1.3.1]{Hov}}.
\end{thm}

\begin{proof}
  It suffices to show that $\bf C$ satisfies the six conditions of
  {\cite[2.1.19]{Hov}} with respect to $I$, $J$ and $W_{\bf C}$.
  Clearly, the first condition holds because $W_{\bf C}$ satisfies the
  two out of three property and is closed under retracts.
  To see that the second and the third conditions hold, let $I_{\bf
    C}$-cell and $J_{\bf C}$-cell be the collections of relative
  $I$-cell and $J$-cell complexes contained in $\bf C$, respectively.
  Since $I_{\bf C}$-cell and $J_{\bf C}$-cell are subcollections of
  the collections of relative $I$-cell and $J$-cell complexes in $\bf
  D$, respectively, the domains of $I$ and $J$ are small relative to
  $I_{\bf C}$-cell and $J_{\bf C}$-cell, respectively.
  The rest of the conditions are verified as follows.  Let $f\colon X
  \to Y$ be a map in $\bf C$.  Since $\eta \colon X \to GX$ is
  isomorphic for $X \in {\bf D}$, $f$ is $I$-injective in $\bf C$ if
  and only if it is $I$-injective in $\bf D$.  Similarly, $f$ is
  $J$-injective in $\bf C$ if and only if it is $J$-injective in $\bf
  D$.  Let $f$ be an $I$-cofibration in $\bf D$.  Then it has the left
  lifting property with respect to all $I$-injective maps in $\bf C$.
  Hence $f$ is an $I$-cofibration in $\bf C$.  Conversely, let $f$ be
  an $I$-cofibration in $\bf C$.  Suppose we are given a commutative
  diagram
  \[
   \begin{CD}
    X
    @>>>
    A
    \\
    @V f VV 
    @V p VV
    \\
    Y
    @>>>
    B
   \end{CD}
  \]
  where $p$ is $I$-injective in $\bf D$.  Then there is a relative
  $I$-cell complex $g \colon X \to Z$ {\cite[2.1.9]{Hov}} such that
  $f$ is a retract of $g$ by {\cite[2.1.15]{Hov}}.  Since $g$ is an
  $I$-cofibration in $\bf D$, there is a lift $Z \to A$ of $g$ with
  respect to $p$.  Then the composite $Y \to Z \to A$ is a lift of $f$
  with respect to $p$.  Therefore $f$ is an $I$-cofibration in $\bf
  D$.  Similarly, $f$ is a $J$-cofibration in $\bf C$ if and only if
  it is a $J$-cofibration in $\bf D$.  Thus we have the desired
  inclusions
  \begin{itemize}
  \item $J_{\bf C}$-cell $\subseteq W_{\bf C} \cap I_{\bf C}$-cof,
  \item $I_{\bf C}$-inj $\subseteq W_{\bf C} \cap J_{\bf C}$-inj, and
  \item either $W_{\bf C} \cap I_{\bf C}$-cof $\subseteq J_{\bf
      C}$-cof or $W_{\bf C} \cap J_{\bf C}$-inj $\subseteq I_{\bf
      C}$-inj.
  \end{itemize}
  Here $I_{\bf C}$-inj and $I_{\bf C}$-cof denote, respectively, the classes of
  $I$-injective maps and $I$-cofibrations in $\bf C$, and similarly
  for $J_{\bf C}$-inj and $J_{\bf C}$-cof. %
  Therefore $\bf C$ is a cofibrantly generated model category by
  {\cite[2.1.19]{Hov}}.

  It is clear, by the definition, that $\bf C$ is finitely generated
  if so is $\bf C$.

  Finally, to prove that $(i,G,\varphi)$ is a Quillen adjunction, it
  suffices to show that $G \colon \bf D \to C$ is a right
  Quillen functor, 
  or equivalently, 
  $G$ preserves $J$-injective maps in $\bf D$ by {\cite[1.3.4]{Hov}} and {\cite[2.1.17]{Hov}}.
  Let $p \colon X \to Y$ be a $J$-injective map in $\bf D$.  Suppose
  there is a commutative diagram
  \[
   \begin{CD}
    A
    @>>>
    GX 
    \\
    @V f VV
    @V Gp VV
    \\
    B
    @>>>
    GY
   \end{CD}
  \]
  where $f \in J$.
  Then we have a commutative diagram
  \[
  \begin{CD}
    A @>>> GX @> \epsilon >> X
    \\
    @V f VV @. @V p VV
    \\
    B @>>> GY @> \epsilon >> Y.
  \end{CD}
  \]
  Since $p$ is $J$-injective in $\bf D$, there is a lift $h \colon B
  \to X$ of $f$.  Thus we have a lift $Gh \circ \eta \colon B \cong
  GB \to GX$ of $f$ with respect to $Gp$.  Therefore $Gp \colon GX
  \to GY$ is $J$-injective in $\bf C$.  Similarly, we can show that $G$
  preserves $I$-injective maps in $\bf C$, and so $G$ preserves
  trivial fibrations in $\bf C$.
Hence $(i,G,\varphi)$ is a Quillen adjunction.
\end{proof}

We turn to the case of pointed categories {\cite[p.4]{Hov}}.
Let $\bf D_{\ast}$ be the pointed category associated with $\bf D$, 
and let $U \colon \bf D_{\ast} \to D$ be the forgetful functor.
We denote by $I_{+}$ and $J_{+}$ the classes of those maps $f \colon X \to Y$ in $\bf D_{\ast}$ such that $Uf \colon UX \to UY$ belongs to $I$ and $J$,
respectively.
Then
we have the following.
(Compare {\cite[1.1.8]{Hov}}, {\cite[1.3.5]{Hov}}, and {\cite[2.1.21]{Hov}}.)
\begin{thm}\label{thm:pointed category}
Let $\bf D$ be a cofibrantly (resp.\ finitely) generated model category,
and let $\bf C$ be a coreflective subcategory satisfying the conditions of Theorem \ref{thm:cofibrant sub}.
Then the pointed category $\bf C_{\ast}$ has a cofibrantly (resp.\ finitely) generated model structure,
with generating cofibrations $I_{+}$ and generating trivial cofibrations $J_{+}$,
such that the induced adjunction $(i_{\ast},G_{\ast},\varphi_{\ast}) \colon \bf C_{\ast} \to D_{\ast}$ is a Quillen adjunction.
\end{thm}
We also have the following Proposition.
\begin{prp}\label{prp:Quillen equivalence} 
%
Suppose $\bf C$ and $\bf D$ satisfy the conditions of Theorem \ref{thm:cofibrant sub}.
Suppose,
further,
that the coreflection arrow $\epsilon \colon GY \to Y$ is a weak equivalence for any fibrant object $Y$ in $\bf D$.
Then the adjunctions $(i,G,\varphi) \colon \bf C \to D$ and $(i_{\ast},G_{\ast},\varphi_{\ast}) \colon \bf C_{\ast} \to D_{\ast}$ are Quillen equivalences.
\end{prp}
\begin{proof}
Let $X$ be a cofibrant object in $\bf C$ and $Y$ a fibrant object in $\bf D$.
Let $f \colon X \to Y$ be a map in $\bf D$.
Then we have $\varphi f=Gf \circ \eta \colon X \cong GX \to GY$.
Since $f$ coincides with the composite $X \xrightarrow{\varphi f} GY \xrightarrow{\epsilon} Y$ and $\epsilon$ is a weak equivalence in $\bf D$,
$\varphi f$ is a weak equivalence in $\bf C$ if and only if $f$ is a weak equivalence in $\bf D$.
It follows by {\cite[1.3.17]{Hov}} that that the induced adjunction $(i_{\ast},G_{\ast},\varphi_{\ast})$ is a Quillen equivalence.
%
\end{proof}
\section{On a model structure of the category \bf NG}
In \cite{KKH} we introduced the notion of numerically generated spaces which turns out to be the same notion as $\Delta$-generated spaces introduced by Jeff Smith (cf.\ \cite{DN}) .
Let $X$ be a topological space.
A subset $U$ of $X$ is numerically open if for every continuous map
$P \colon V \to X$, where $V$ is an open subset of Euclidean space,
$P^{-1}(U)$ is open in $V$.
Similarly, $U$ is numerically closed if for every such map $P$,
$P^{-1}(U)$ is closed in $V$.
A space $X$ is called a numerically generated space 
if every numerically open subset is open in $X$.

Let $\bf NG$ denote the full subcategory of $\bf Top$ consisting of numerically generated spaces.
Then the category $\bf NG$ is 
cartesian closed {\cite[4.6]{KKH}}.
To any $X$ we can associate the numerically generated space topology,
denoted $\nu X$,
by letting $U$ open in $\nu X$ if and only if $U$ is numerically open in $X$.
Therefore we have a functor $\nu \colon \bf Top \to NG$ which takes $X$ to $\nu X$.
Clearly, the identity map $\nu X \to X$ is continuous.
By the results of {\cite[\S 3]{Hov}} the following holds.
\begin{prp}
The functor $\nu \colon \bf Top \to NG$ is a right adjoint to the inclusion functor $i \colon \bf NG \to Top$,
so that $\bf NG$ is a coreflective subcategory of $\bf Top$.
\end{prp}
A continuous map $f \colon X \to Y$ between topological spaces is called a weak homotopy equivalence in $\bf Top$ if it induces an isomorphism of homotopy groups
\[
f_{\ast} \colon \pi_{n}(X,x) \to \pi_{n}(Y,f(x))
\]
for all $n > 0$ and $x \in X$.
Let $I$ be the set of boundary inclusions $S^{n-1} \to D^{n},\ n \geq 0,\ J$ the set of inclusions $D^{n} \times \{0\} \to D^{n} \times I,$
and $W_{\bf Top}$ the class of weak homotopy equivalences.
The standard model structure on $\bf Top$ can be described as follows.
\begin{thm}[{\cite[2.4.19]{Hov}}]
There is a
finitely generated model structure on $\bf Top$ with $I$ as the
set of generating cofibraitons, $J$ as the set of generating trivial
cofibrations, and $W_{\bf Top}$ as the class of weak equivalences.
\end{thm}
The category $\bf NG$ is complete and cocomplete by {\cite[3.4]{KKH}}.
A space $X$ is numerically generated if and only if $\nu X=X$ holds.
Thus the unit of the adjunction $\eta \colon X \to \nu  X$ is a natural homeomorphism.
Moreover, since CW-complexes are numerically generated spaces by {\cite[4.4]{KKH}},
the classes $I$ and $J$ are contained in $\bf NG$.
Let $W_{\bf NG}$ be the class of maps $f \colon X \to Y$ in $\bf NG$ which is a weak equivalence in $\bf Top$.
Since the coreflection arrow $\nu Y \to Y$,
given by the identity of $Y \in \bf Top$,
is a weak equivalence (cf.\ {\cite[5.4]{KKH}}),
we have the following by Theorem \ref{thm:cofibrant sub} and Proposition \ref{prp:Quillen equivalence}.
\begin{thm}\label{thm:Quillen equivalence in NG}
The category $\bf NG$ has a finitely generated model structure with $I$ as the set of generating cofibrations,
$J$ as the set of generating trivial cofibrations,
and $W_{\bf NG}$ as the class of weak equivalences.
Moreover the adjunction $(i,\nu,\varphi) \colon \bf NG \to Top$ is a Quillen equivalence.
\end{thm}
We turn to the case of pointed spaces.
Let $\bf Top_{\ast}$ be the category of pointed topological spaces.
By {\cite[2.4.20]{Hov}}, there is a finitely generated model structure on the category $\bf Top_{\ast}$,
with generating cofibrations $I_{+}$ and generating trivial cofibrations $J_{+}$.
Then we have the following by Theorem \ref{thm:pointed category} and Proposition \ref{prp:Quillen equivalence}.
\begin{crl}\label{crl:Quillen of pointed}
There is a finitely generated model structure on the category $\bf NG_{\ast}$ of pointed numerically generated
spaces, with generating cofibrations $I_{+}$ and generating trivial cofibrations $J_{+}$.
Moreover, the inclusion functor $i_{\ast} \colon \bf NG_{\ast} \to Top_{\ast}$ is a Quilen equivalence.
\end{crl}
\begin{rmk}(1)
  The argument of Theorem \ref{thm:Quillen equivalence in NG} can be applied to the subcategories $\bf K$ of $k$-spaces and $\bf T$ of compactly generated spaces.
  Similarly,
  the argument of Corollary \ref{crl:Quillen of pointed} can be applied to the pointed categories $\bf K_{\ast}$ and $\bf T_{\ast}$.
  Compare [2.4.28], [2.4.25], [2.4.26] of \cite{Hov}.

  (2)
  Let $\bf Diff$ be the category of diffeological spaces (cf.\
  \cite{Zem}).  In \cite{KKH} we introduced a pair of functors $T:\bf
  Diff \to Top$ and $D:\bf Diff \to Top$, where $T$ is a left adjoint
  to $D$, and showed that the composite $TD$ coincides with $\nu
  \colon \bf Top \to \bf NG$.  Thus $\bf NG$ can be embedded as a full
  subcategory into $\bf Diff$.
  It is natural to ask whether $\bf Diff$ has a model category
  structure with respect to which the pair $(T,D)$ gives a Quillen
  adjuntion between $\bf Top$ and $\bf Diff$.

  Let $I$ be the unit interval, and let $\lambda \colon \mathbf{R} \to
  I$ be the smashing function, that is, a smooth function such that
  $\lambda(t) = 0$ for $t \leq 0$ while $\lambda(t) = 1$ for $t \geq
  1$.  Let $\tilde{I}$ denote the unit interval equipped with the
  quotient diffeology $\lambda_*(D_{\mathbf R})$, where $D_{\mathbf
    R}$ is the standard diffeology of $\mathbf{R}$.
  In \cite{KH} we introduce a finitely generated model category
  structure on $\bf Diff$ with the boundary inclusions $\partial
  \tilde{I}^{n-1} \to \tilde{I}^{n}$ as generating cofibrations, and
  with the inclusions $\partial \tilde{I}^{n-1} \times \tilde{I} \cup
  \tilde{I}^{n} \times \{0 \} \to \tilde{I}^{n} \times \tilde{I}$ as
  generating trivial cofibrations.  Its class of weak equivalences
  consists of those smooth maps $f \colon X \to Y$ inducing an
  isomorphism $f_* \colon \pi_{n}(X,x_{0}) \to \pi_{n}(Y,f(x_{0}))$
  for every $n \geq 0$ and $x_{0} \in X$.  Here, the homotopy set
  $\pi_n(X,x_0)$ is defined to be the set of smooth homotopy classes
  of smooth maps $(\tilde{I}^n,\partial\tilde{I}^n) \to (X,x_{0})$.

  It is expected that with respect to the model structure on $\bf
  Diff$ described above, the pair $(T,D)$
  induces a Quillen adjunction between $\bf Top$ and $\bf Diff$.
\end{rmk}
\section*{Acknowledgements.}
I would like to express my sincere gratitude to my supervisor Kazuhisa Shimakawa.
He introduced me to the project of building a homotopy theory on the category $\bf NG$.
This project was not completed without him.
He carefully read this paper,
helped me with the English and corrected many errors.
In order that I might acquire a doctor's degree,
he had supported me for a long time.

Next I would like to thank Referee.
He proposed that there exists a general framework from proofs of Theorem \ref{thm:Quillen equivalence in NG} and Corollary \ref{crl:Quillen of pointed}.
As a result,
we have Theorem \ref{thm:cofibrant sub}, Theorem \ref{thm:pointed category}, and Proposition \ref{prp:Quillen equivalence}.

\end{document}